 \documentclass[a4paper,12pt,oneside,reqno]{amsart}

\usepackage[utf8]{inputenc}
\usepackage[T1]{fontenc}
\usepackage{hyperref}
\usepackage[headinclude,DIV13]{typearea}
\usepackage{amssymb,amsmath,amsthm}
\usepackage[usenames,dvipsnames]{xcolor}
\definecolor{violet}{rgb}{0.6,0.4,0.8}
\usepackage{enumerate,multicol,longtable,array}
\usepackage{graphicx} 
\usepackage{caption}
\usepackage{mathrsfs}
\usepackage{xcolor}
\usepackage{enumerate}
\theoremstyle{plain}
\usepackage{txfonts}
 
\newtheorem{thm}{Theorem}
\newtheorem{defn}{Definition}
\newtheorem{lemma}{Lemma}

\newtheorem{rk}{Remark}
\newtheorem{cor}{Corollary}

\newtheorem {prop}[thm]{Proposition}

\theoremstyle{definition}

\theoremstyle{remark}

\numberwithin{equation}{section}

\renewcommand{\P}{\mathbb{P}}
\newcommand{\R}{\mathbb{R}}

\newcommand{\N}{\mathbb{N}}
\newcommand{\Z}{\mathbb{Z}}

\def\N{{\mathbb N}}

\def\Z{{\mathbb Z}}
\def\R{{\mathbb R}}
\def\P{{\mathbb P}}

\newcommand{\F}{{\mathcal F}}
\newcommand{\G}{{\mathcal G}}

\numberwithin{equation}{section}

\graphicspath{ {fig/} }
\def\be{\begin{equation*}}
\def\ee{\end{equation*}}
\def\best{\begin{equation*}}
\def\eest{\end{equation*}}
\begin{document}

\title[Extremal inhomogeneous Gibbs states for spin 
models on trees]{Extremal inhomogeneous Gibbs states \\for SOS-models and finite-spin 
models on trees }

\author{Loren Coquille}
\address{Loren Coquille,
	Univ. Grenoble Alpes, CNRS, Institut Fourier, F-38000 Grenoble, France}
\email{loren.coquille@univ-grenoble-alpes.fr}

\author{Christof Külske}
\address{Christof Külske,
	Ruhr-Universit\"at Bochum, Fakult\"at f\"ur Mathematik, D44801 Bochum, Germany}
\email{Christof.Kuelske@ruhr-uni-bochum.de}

\author{Arnaud Le Ny}
\address{Arnaud Le Ny,
	LAMA UMR CNRS 8050, UPEC, Universit\'e Paris-Est, 61 Avenue du G\'en\'eral de Gaulle,  94010 Cr\'eteil cedex, France}
\email{arnaud.le-ny@u-pec.fr}


\thanks{C.K. and A.L.N. thank Labex B\'ezout (ANR-10-LABX-58)
and Laboratory LAMA (UMR CNRS 8050) at Universit\'e Paris Est Cr\'eteil (UPEC) for 
the support of a visit of one month of C.K. which was crucial for the realization of this work. Research of ALN have also been supported by the CNRS IRP B\'ezout-Eurandom ``Random Graph, Statistical Mechanics and Networks''.}

\maketitle

\begin{abstract}
    We consider $\Z$-valued $p$-SOS-models {with nearest neighbor interactions of the form $|\omega_v-\omega_w|^p$}, 
    and {finite-spin ferromagnetic models} on regular trees. 
    This includes the classical SOS-model, the discrete Gaussian model and the Potts model. 
    
    {We exhibit a family of extremal inhomogeneous (i.e. tree automorphism non-invariant) Gibbs measures arising as low temperature perturbations of {{ground states}  (local energy minimizers),  which have a sparse enough set of broken bonds together with uniformly bounded increments along them}.
    These low temperature states in general do not possess any symmetries of the tree. }

    {This generalises the results of Gandolfo, Ruiz and Shlosman \cite{GRS12} about the Ising model, and shows that the latter behaviour is robust. 
    We treat three different types of extensions:
       non-compact state space gradient models,
       models without spin-symmetry, and
       models in small random fields.}

    
    We give a detailed construction  and 
    full proofs of the extremality of the low-temperature states 
    in the set of all Gibbs measures,  
    analysing excess energies relative to the ground states, 
    convergence of low-temperature expansions, and properties of cutsets.
\end{abstract}

{\small{
\tableofcontents
}}

\section{Introduction}	

{
Amongst probability measures on lattice spin systems, Gibbs states on trees have special properties,  widely studied for almost fifty years by now, see e.g. \cite{Spi75, Pres74, zach83, hig77, HOG, Lyons89}. 

{In the early nineties, Blekher and Ganikhodgaev \cite{BLG91}, following a strategy already proposed by Higuchi in the late seventies \cite{hig77}, show that the Ising model (in zero external field, on 
regular trees) possesses uncountably many interface states, which are extremal and non translation invariant, as soon as $T<T_c$.
In 2008, Rozikov and Rakhmatullaev \cite{RR08} exhibit non-translation invariant measures corresponding to subgroups in the group representation of the Cayley tree, the so-called "weakly periodic" Gibbs measures. 
These states can be thought of generalizations of Dobrushin states from \cite{Dob72,vB75}, but with many interfaces, possibly countably infinitely many.
}

{In 2012, Gandolfo, Ruiz and Shlosman \cite{GRS12}, exhibit a
rich family of extremal inhomogeneous (i.e. tree automorphism non-invariant) Gibbs measures arising
as low temperature perturbations of {{ground states}  (local energy minimizers),  which have a sparse enough set of broken bonds}, see also \cite{GRS15}}. 
 These low temperature states in general do not possess any symmetries of the tree. 

The aim of the present paper is {to show that the latter behaviour is robust. 
We prove it to hold in three different types of extensions of the Ising model:
\begin{itemize}
    \item non-compact state space gradient models
    \item models without spin-symmetry 
    \item models in small random fields 
\end{itemize}
}

The main objective  is to study integer-valued 
gradient models, but we also derive similar results for general finite-alphabet spin models of ferromagnetic type, including 
the Potts model. 

Gradient models belong to a very active field of research, also widely studied in the literature, either 
on lattices to model effective interfaces, see e.g. \cite{BK94, FS97, DGI00, Shef05, BKS07, vEK08, BS11, KL14, CK12, DHP21, LO21} or more specifically on trees \cite{HKLNR19, HK21, HK22, Roz13}.

{In \cite{HK21}, Henning and Kuelske prove, for very general classes 
of gradient interactions, assuming strong enough coupling, that there exist homogeneous Gibbs states which are strongly localised around one given height.}
The method of proof is analytic in character and based on finding fixed points of a suitable 
non-linear operator in a $l^p(\Z)$-sequence space of so-called boundary laws, 
starting from the description of Zachary \cite{zach83}, see also \cite{BeKiKu20, EnErIaKu12}. 
Strong coupling of the interactions allows to prove that the relevant operator is a contraction. 
A variant of that contraction method is used to prove the 
existence of a different type of consistent measures, namely 
\text{delocalized} gradient Gibbs measures (with unbounded height-fluctuations), still in strong coupling 
regimes. 
{Interestingly, coexistence of delocalized and localized states for the same interaction parameters is possible.}
Using dynamical systems ideas with an analysis of the unstable manifold 
around the free state, {special types of }inhomogenous gradient states are constructed, which still possess some rotation invariance, see \cite{HK22}. 

The present work first treats   general $p$-SOS models, {with nearest neighbor interactions, where the interaction
along} an edge $(v,w)$ of the tree is of the form $\beta |\omega_v-\omega_w|^p$, with a fixed interaction 
exponent $p\in (0,\infty)$, and sufficiently large inverse temperature $\beta\in (0,\infty)$. 
The cases $p=1$ (the classical SOS-model) and $p=2$ (discrete Gaussian, DGFF) are the most popular choices, 
see e.g. \cite{V06, BPR22-1, BPR22-2} and references therein. 
Our approach to inhomogeneous states is completely independent from the latter two results \cite{HK21,HK22}, 
based on the boundary law formalism. 
Instead of it, we rigorously develop low-temperature expansions around a suitable class of ground states {generalising the ones initially introduced in \cite{GRS12} in the case of the Ising model}. {These ground states do not possess any symmetries as rotation invariance in general.}

More precisely, in this enlarged framework, we generalise the definition of contours (or low temperature excitations), including increment sizes {(which were not needed for Ising)} and give a new proof of the control on the excess energy created by a low temperature excitation above a given non-homogeneous configuration. When this configuration has a sparse set of broken bonds, together with bounded increments, the excess energy control allows to conclude that it is a stable local ground state provided the degree of the tree is large enough.
We provide rigorous proofs of :
\begin{itemize}
	\item tightness and convergence of finite volume measures {with boundary conditions given by these non-hommogenous configurations}
	
	\item extremality of the low temperature states, derived from cluster expansion and cutset properties. 
	
	\item the stability of the excess-energy control under the addition of small local field terms, which provides the extension of our results to models in small random fields. 
\end{itemize}
Note that our contours have empty interior, a tree-specific property which provides more control in Peierls-type estimates and low-temperature expansions, also around inhomogeneous {ground states}. This allows to prove more refined results than on the lattice, where versions of Pirogov-Sinai theory would be necessary to treat situations without symmetry in spin space, even when there is spatial homogeneity, see e.g. Chapter 7 of  \cite{FV-book}. This particularity allows us to prove the stability of these inhomogeneous ground states, and existence of well-defined {infinite volume} limits with the required decorrelation properties, by combining statistical mechanics technics (as cluster expansions) with probabilistic methods (cutsets, Fourier transforms, etc.). \\
}

{The paper is organised as follows.}
In Section \ref{sec-results}, we {state our results on} the stability of some {ground states} at low temperature. We first treat general $p$-SOS models {(Theorem \ref{Theorem1})}, and afterwards come to general finite-alphabet models, including the Potts model {(Theorem \ref{Theorem2})}. 

In Section \ref{sec-excess-energy} we provide the definition of contours, as well as the proof of their excess energy estimate (Lemma \ref{Lem1}).

In Section \ref{sec-properties}, we consider these contours as polymers to perform cluster expansions within the framework of Bovier-Zahradn\'ik \cite{BZ00} and study the  low-temperature states. We use the estimates they provide in addition to the convergence of the expansions to prove Theorem \ref{Theorem1} and Theorem \ref{Theorem2}. We get an exponential control of the polymer weights (Proposition \ref{Prop3}), convergence of finite-dimensional marginals via Fourier transforms (Lemma \ref{Lem3}), quantitative tightness in the unbounded spins case (Section \ref{sec-tightness}), DLR-property of the limiting measures (Section \ref{sec-DLR}), identifiability of the different low-temperatures phases obtained from sparse {ground states}  (Section \ref{sec-DLR}). {Finally, {in Section \ref{sec-cutset}}, we derive cutset properties as well as  correlation decay for events of polymer type, that eventually lead to extremality.}

In Section \ref{Section5} we describe applications of our theory 
to existence of extremal states for inhomogeneous locally perturbed models. This includes the random field 
Potts model, and the $p$-SOS model in random fields 
and in random media.

\section{Definitions and main results} \label{sec-results}

{Let $\mathcal{T}^d=(V,E)$ denote the Cayley tree of order $d$, on which any vertex $i\in V$ has exactly $d+1$ neighbors. 
To any vertex $i\in V$, we attach a spin, which is a random variable $\sigma_i$ taking values in $\Omega_0$.
The spin space $\Omega_0$ we consider will be either the discrete set $\Z_q=\{0,\dots, q-1\}$, for $q \in \{2,3,\ldots\}$, or the unbounded countable set $\Z$,
equipped with a product $\sigma$-algebra $\mathcal{E}=\mathcal P(\Omega_0)$.
We are interested in probability measures on the product space $(\Omega,\mathcal F)=(\Omega_0^V, \mathcal{E}^{\otimes V})$.} 
For any subset $\Lambda\subset V$ we define $\omega_\Lambda=(\omega_v)_{v\in\Lambda}$.
For any subset $W\subset V$ we denote by $\mathcal{F}_W$ the 
sigma algebra generated by the variables $(\sigma_i)_{i\in W}$.
If $\Lambda \subset V$ is a finite subset, we write $\Lambda\Subset V$.

We introduce an interaction potential $\Phi$ and consider equilibrium states to be Gibbs measures built with the DLR framework, see e.g. \cite{HOG}: they are the probability measures $\mu$ consistent with the Gibbsian specification $\gamma^\Phi$ in the sense that a version of their conditional probabilities w.r.t. the outside of any finite set $\Lambda$ of the tree is given by the corresponding element of the Gibbs specification $\gamma_\Lambda^\Phi$, that is
$$
\forall \Lambda \Subset V,\; \forall \omega_\Lambda \in \Omega_\Lambda,\; \mu[\sigma_\Lambda=\omega_\Lambda \mid \mathcal{F}_{\Lambda^c}](\cdot) = \gamma_\Lambda^\Phi(\omega_\Lambda \mid \cdot), \; \mu-a.s.
$$

where the elements of the Gibbs specification $\gamma^\Phi=\gamma^\Phi(\beta)$ are the probability kernels $\gamma_\Lambda^\Phi$ from  $\Omega_{\Lambda^c}$  
to $\mathcal{F}_\Lambda$ 
defined for all finite $\Lambda$   as
$$
\gamma_\Lambda^\Phi(\omega_\Lambda \mid \tau_{\Lambda^c})=  \frac{1}{Z_\Lambda^\tau}e^{-\beta H_\Lambda(\omega_{\Lambda} \tau_{\Lambda^c})}.
$$

{The partition function $Z_\Lambda^\tau$ is the usual normalization constant for a fixed boundary condition $\tau$, at finite volume $\Lambda$, and the Hamiltonian $H_\Lambda^\Phi$ with boundary condition $\tau$ is there provided by
$
H_\Lambda(\omega_\Lambda  \tau_{\Lambda^c}) = \sum_{A \cap \Lambda \neq \emptyset} \Phi_A( \omega_\Lambda  \tau_{\Lambda^c})
$ where $\omega_\Lambda  \tau_{\Lambda^c}$ denotes the concatenation of $\omega_\Lambda$ and $\tau_{\Lambda^c}$. 
We sometimes shortly write $H$ for the Hamiltonian with free boundary conditions:
\begin{equation}\label{freeH}
H(\omega) = H_\Lambda^f(\omega):=\sum_{A \subset \Lambda} \Phi_A(\omega_\Lambda).
\end{equation}
The ferromagnetic potentials we consider are nearest-neighbor potentials and will be generically denoted by $\Phi$. Pairs $\{i,j\}\in E$ of nearest-neighbors are  written $i \sim j$.
In the case where $\Omega_0$ is unbounded, we consider $p$-SOS models, where the potential is given for any $p>0$ by
\begin{equation}\label{model-p-sos}
\Phi_{ij}(\omega) = | \omega_i - \omega_j | ^p\quad\text{for all}\quad i\sim j
\end{equation}
and 0 otherwise, where $\mid \cdot \mid$ is the absolute value. 
In the case where $\Omega_0=\Z_q$, we consider any nearest neighbor model of the form
\begin{equation}\label{model-finite-spin}
\Phi_{ij}(\omega) = \sum_{k,\ell=0}^{q-1} u_{k,\ell}\mathbf{1}_{\omega_i =k, \omega_j=\ell}
\text{ where } {\forall k,\ell\in\Z_q, u_{k,\ell}\geq0}\text{ and } u_{k,k}=0.
\end{equation}
The latter includes the $q$-state Potts model, for which
\begin{equation}\label{model-potts}
\Phi_{ij}(\omega) =  \mathbf{1}_{\omega_i \neq \omega_j}.
\end{equation}
}

{For a given set of edges $D \subset E$, 
and a given vertex $v\in V$ we define $d_D(v)$ to be the number of bonds in $D$ which are incident to $v$. Then, we define the number :
$$d_D=\max_{v\in V} d_D(v).$$}

\subsection{{p}-SOS models}

Existence of phase transitions on trees with homogeneous phases holds for very general interactions at low-temperature \cite{HK21,HK22}. However, we investigate here the 
low-temperature stability of non translation-invariant (inhomogeneous) ground states, defined as follows below (see an example in Figure \ref{fig-GS}), {and show how they are related to infinite-volume Gibbs measures.}
{Our first main result reads then: }

\begin{thm}\label{Theorem1} {Let $p>0$ and consider the $p$-SOS models (\ref{model-p-sos}) {on the Cayley tree of degree $d$}.} {Let $d_{\max},M\in \N_0$, where $d_{\max}\leq d$.} Define $\G^0=\G^0({d_{\max}},{d},M)$ to be the set of configurations $\omega^0\in \Z^V$ which satisfy the following
{sparsity requirement on the set of broken edges and have uniformly bounded spin increments along them:}
\begin{enumerate}
\item\label{sparsity} The set of broken edges 
$D:=\{\{v,w\}\in E:  \omega^0_v\neq\omega^0_w\}$ {is such that $d_D\leq d_{\max}$.}
\item\label{bdd-incr}  All increments are uniformly bounded by $M$: 
$\max_{v\sim w}|\omega^0_v-\omega^0_w|\leq M$.
\end{enumerate}

Then, for each interaction exponent $p>0$, each maximal increment size
$M\in \N_0$, and each {maximal internal degree $d_{\max}\in \N_0$} there is a minimal degree 
$d(p,M,{d_{\max}})$ such that for all degrees $d\geq d(p,M,{d_{\max}})$ the following holds: 

\smallskip
There exists a finite $\beta_0=\beta_0(d,p,{d_{\max}},M)$ such that 
for all $\beta\geq \beta_0$ for  
the $p$-SOS model on the regular tree of degree $d$, 
there is a family of Gibbs measures $(\mu^{\omega^0}_{\beta})_{\omega^0\in \G^0}$ with the properties
\begin{enumerate}
\item $\omega^0\neq \tau^0$ implies $\mu^{\omega^0}_{\beta}\neq \mu^{\tau^0}_{\beta}$.
\item  The measures $\mu^{\omega^0}_{\beta}$ are extremal in the set of all Gibbs measures.
\item  $\mu^{\omega^0}_{\beta}$ concentrates around  $\omega^0$ in the sense that 
{there exist two positive constants $c,C$ such that} 
for any $v\in V$ and
for all increments $k\in \Z$, 
\begin{equation}
\begin{split}
&\mu^{\omega^0}_\beta(\sigma_v-\omega_v^0=k)\leq C e^{-c \beta |k|^p }.
\end{split}
\end{equation}
\end{enumerate}
\end{thm}

{\bf Remark.} 
Further properties will be derived in the explicit construction, see below. Moreover, the assumptions \ref{sparsity} and \ref{bdd-incr} of Theorem \ref{Theorem1} can be replaced by the more general assumption 
\eqref{generalsparsity} below which mixes geometric sparsity and boundedness of heights.

\subsection{Finite-spin ferromagnetic models}

We have an analogous theorem in the situation of finite-alphabet models \eqref{model-finite-spin}
with generalized ferromagnetic interactions $\Phi \geq 0$ in the following sense.  

\begin{thm}\label{Theorem2}{Let $q\in\N_0$ and consider the $q$-spin model (\ref{model-finite-spin}) {on the Cayley tree of degree $d$}. {Let $d_{\max}\in \N_0$, with $d_{\max}\leq d$.}
Put $u:={\min}_{k\neq \ell}u_{k,\ell}$ and $U:=\max_{k,\ell}u_{k,\ell}$.}
Define $\G_q^0=\G_q^0({d_{\max},{d}})$ to be the set of configurations 
$\omega^0\in \Z^V$ whose set of broken bonds 
$D=\{\{v,w\}\in E:\omega^0_v\neq\omega^0_w\}$ {is such that $d_D\leq d_{\max}$}.  

Then, under the following geometric sparsity condition on the set of broken bonds
\begin{equation}\label{sparsity-qspin}
(d-1)u>d_{\max}( U+u)
\end{equation}
there exists a finite $\beta_0=\beta_0(d,q,u,U,{d_{\max}}) \; {>0}$ such that 
for all $\beta\geq \beta_0$ 
there is a family of Gibbs measures $(\mu^{\omega^0}_{\beta})_{\omega^0\in \G^0}$ with the following properties:
\begin{enumerate}
\item $\omega^0\neq \tau^0$ implies $\mu^{\omega^0}_{\beta}\neq \mu^{\tau^0}_{\beta}$.
\item  The measures $\mu^{\omega^0}_{\beta}$ are extremal in the set of all Gibbs measures.
\item  $\mu^{\omega^0}_{\beta}$ concentrates around  $\omega^0$ in the sense that {there exist two positive constants $c,C$ such that} for any vertex $v \in V$,  
\begin{equation}
\begin{split}
&\mu^{\omega^0}_{\beta}(\sigma_v=\omega_v^0)\geq 1-C e^{-c \beta }.
\end{split}
\end{equation}

\end{enumerate}
\end{thm}
{\bf Remark:} {In the case of the Potts model (\ref{model-potts}), we have $u=U=1$ }and thus the sparsity assumption \eqref{sparsity-qspin} on $d_{\max}$ becomes 
$2 d_{\max}<d-1$ which coincides with the sparsity requirement for the Ising model 
given in \cite{GRS12}. See Figure \ref{fig-GS} for an example.


{\color{black}

\section{Excess energy for sparse ground states}\label{sec-excess-energy}
In this section, we derive useful lower bounds on {excess energies} in our models, which are the starting point of the low-temperature expansions and extensions. 
Similar estimates were obtained for the Ising model in \cite{GRS12} using 
induction over the size of the contours. Here we follow a different non-inductive 
approach which provides {useful bounds in the case of unbounded spins.} 

\smallskip

Let us start with the introduction of contours as labelled contours, namely as 
pairs of supports $\gamma$ 
and spin configurations $\omega_{\gamma}$ on these supports.

\begin{defn}Let $\omega^0\in \Omega_0^V$ be a fixed reference 
configuration. 
A contour for the general spin configuration $\omega\in \Omega_0^V$ relative to $\omega^0$
is a pair $\bar \gamma=(\gamma, \omega_{\gamma})$ where the support {
$\gamma=\{v\in V: \omega_v\neq \omega_v^0\}$ is 
a connected component of the set of incorrect points for 
$\omega$ (with respect to $\omega^0$), and {$\omega_\gamma=(\omega_v)_{v\in\gamma}$}.}
See Figure \ref{fig-contour}.
\end{defn}

The contour definition 
above generalizes the one of \cite{GRS12} for the Ising model, in the 
sense that it {also encodes the} spin configuration on the support. This definition facilitates to relate 
probabilities of the occurrence of given local patterns to suitable contour sums. 
Due to their tree-nature, our contours always have no interior components {of their complement}, which allows to avoid symmetry {requirements} or spin-flip considerations in applying Peierls-type arguments {and 
expansions}. 

Note moreover that for each $\omega$ for which $\bar \gamma$ {is a contour} 
we must have $\omega_{\partial \gamma}=\omega^0_{\partial \gamma}$, i.e.  
the spin configuration must take the values of the ground state in the outer boundary 
of the contour support.

\begin{figure}
  \begin{center}
  \includegraphics[width=7cm]{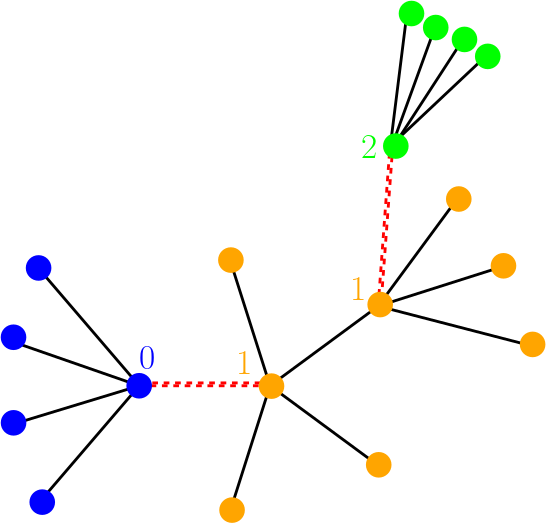}\quad
  \includegraphics[width=7cm]{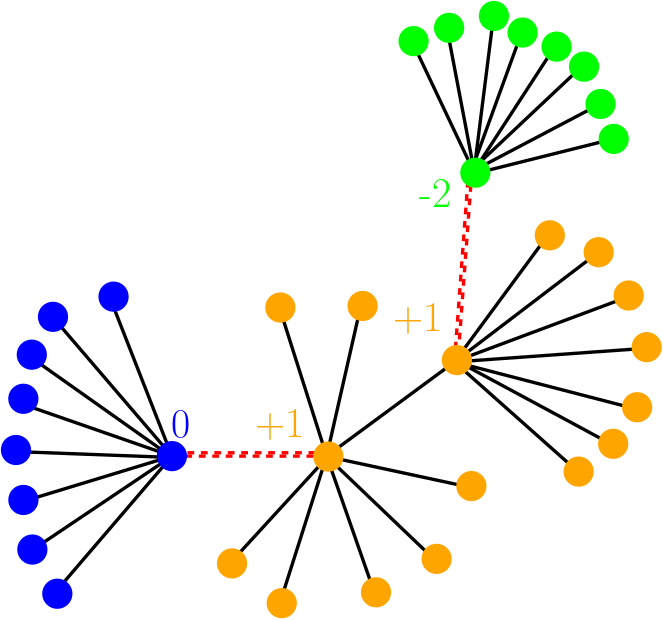}
  \caption{ {Ground state configurations with $d_{\max}=1$ (same colors denote same spin values, dashed lines denote broken bonds): 
  (Left) for the 3-Potts model. The above configurations are ground states for $d\geq4$ (see the proof of Corollary \ref{cor-finite-spin}).
  (Right) for the SOS model, with $p=1,M=3$. The above configuration is a ground state for $d\geq8$ (see the proof of Corollary \ref{cor}). }}\label{fig-GS}
  \end{center}
\end{figure}

\begin{figure}
  \begin{center}
  \includegraphics[width=7cm]{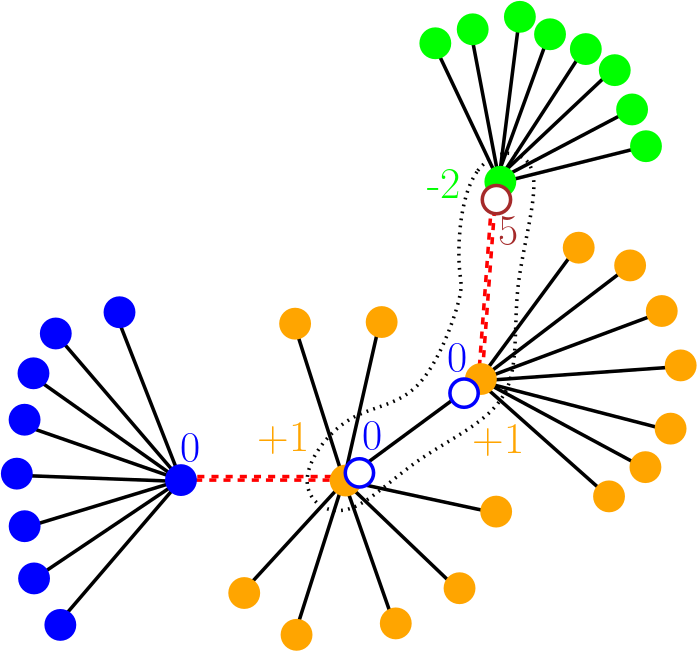}
  \caption{A contour $\bar\gamma=(\gamma,\omega_\gamma)$ relative to the ground state  $\omega^0$ of Figure \ref{fig-GS}. The configuration $\omega^0$ is depicted by full circles. The support $\gamma$ consists of 3 sites surrounded by a dotted line. The configuration $\omega_\gamma$ is depicted by open circles.}\label{fig-contour}
  \end{center}
\end{figure}

\subsection{Stable inhomogeneous ground states for the {p}-SOS models}


Consider the homogeneous $p$-SOS-models defined in {\eqref{model-p-sos}}, 
for $d\geq 2$. 

\smallskip

{ 
{
\begin{defn}
  We say that a configuration $\omega^0\in \Omega$ is stable with 
stability constant $c>0$ if for all configurations $\omega\in \Omega$ differing from $\omega^0$ at finitely many sites, the excess energy relative to $\omega^0$
satisfies the lower bound 
\begin{equation}
\begin{split}
\label{XSpSOS}
	&H(\omega)- H(\omega^{0}) \geq c\sum_{v}|\omega_v-\omega^0_v|^p.
\end{split}
\end{equation}
\end{defn}
}
In particular, all stable configurations are ground states in the usual sense that finite volume perturbations raise the energy.

This notion will allow to perform the large-$\beta$ expansions around stable $\omega^0$ 
of Section \ref{sec-properties}, as we will see. 
To formulate the lemma on the excess energy and develop a viable criterion 
on the type of ground states which are stable, consider a pair of configurations $(\omega,\omega^0)$ 
which differ on a contour $\bar\gamma$. }The following notations are useful. 
Describe the geometric part $\gamma$ 
of a contour to be a finite subtree rooted at the origin $0$. We think of it
embedded into the full tree which we describe as 
rooted tree which has $d+1$ offspring at the origin, but offspring $d$ 
at all other sites. By homogeneity of the tree and of the potential, there is no loss in doing so. 
We write $w \leftarrow v$ if 
$w$ is a child (or offspring) of $v$ on the full tree relative to the chosen root. 
We write $k_v$ for the number of children of $v$ {in} the contour $\gamma$. 
For $v\neq 0$ we have  $k_v\in \{0,1,\dots, d\} $, while $k_0\in \{0,\dots,d+1\}$. 

\smallskip

Now we present the lemma on the excess energy of a spin configuration $\omega$ 
relative to a general ground state $\omega^0\in \Z^V$. 

\begin{lemma}\label{Lem1}
{Let $p\in (0,\infty)$, $\omega^0\in \Z^V$, and $\omega=\omega_\gamma\omega^0_{\gamma^c}$ such that $\bar\gamma=(\gamma, \omega_\gamma)$ is a contour with respect to the fixed configuration $\omega^0$.}

Then the excess energy satisfies
	\begin{equation}
\begin{split}
\label{XSpSOS2}
	&H(\omega)- H(\omega^{0}) 
	=\sum_{v \in \gamma, w\leftarrow v}\Bigl( |\omega_v-\omega_w|^p -|\omega^0_v-\omega^0_w|^p \Bigr)\cr
&	\geq (d c_p^2 -1)\sum_{v\in \gamma}|\omega_v-\omega^0_v|^p
	-(c_p +1)\sum_{v \in \gamma, w\leftarrow v}|\omega^0_v-\omega^0_w|^p
\end{split}
\end{equation}
where $c_p=\min\{2^{1-p},1\}$. 
\end{lemma}

\begin{cor}\label{cor} For each interaction exponent $p>0$, each maximal increment size
$M\in \N_0$, and each maximal {internal} degree 
${d_{\max}}\in \N_0$ there is a minimal degree 
$d(p,M,{d_{\max}})$ such that for all $d\geq d(p,M,{d_{\max}})$
all configurations {$\omega^0\in \G^0({d_{\max}},M)$} are stable ground states.
\end{cor}


\begin{proof}[Proof of Corollary \ref{cor}] 
  By Lemma \ref{Lem1}, 
  \begin{equation}
    \begin{split}
      &H(\omega)- H(\omega^{0}) 
      \geq (d c_p^2 -1-(c_p +1) d_{\max} M^p)\sum_{v\in \gamma}|\omega_v-\omega^0_v|^p.
    \end{split}
    \end{equation}

    Thus by taking
$d(p,M,{d_{\max}})= 1+ \lfloor c_p^{-2} + (c_p^{-1}+
c_p^{-2}){d_{\max}}  M^p \rfloor,$
  this  implies stability with the constant $c:=d c_p^2 -1-(c_p +1) d_{\max} M^p>0$.
\end{proof}

\smallskip

{\bf Remark. } We may also work with a {more general} mixed sparsity requirement 
on $\omega^0$ of the form 
	\begin{equation}
\begin{split}
\label{generalsparsity}
	&(d c_p^2 -1)
	>(c_p +1)\sup_{v}\sum_{ w\sim  v}|\omega^0_v-\omega^0_w|^p
\end{split}
\end{equation}
{which follows from the right hand side of \eqref{XSpSOS2}, and which will provide good bounds to ensure convergence of the cluster expansion, when multiplied with a sufficiently large $\beta$}, see below. 

\smallskip

{\bf Remark. }  {In the particular case of} ground state increments bounded in modulus  by $M=1$, 
and sparse set of broken bonds with $d_D=1$, we get the minimial degree of 
$d(p,1,1)=4$, as for the Ising model in \cite{GRS12}, for all $p\leq 1$. 
For the discrete Gaussian {(2-SOS model)} we get  $d(2,1,1)=11$.

\proof
Let us write 
$$ 
s_v:=\omega_v-\omega^0_v
$$
for the deviation of the spin from the ground state. As $\gamma$ is the support 
of the contour we have $s_v\neq 0$ 
for $v\in \gamma$ and $s_w=0$ for $w\not\in \gamma$. 
Note that 
\begin{equation}
\begin{split}\label{leavelower}
\inf_{s\in \R}(|s+t|^p+|s|^p)=c_p |t|^p
\end{split}
\end{equation}
where, by scaling $c_p=\inf_{s\in \R}(|s+1|^p+|s|^p)$.

{For the standard SOS-model with 
$p=1$ this becomes the triangle inequality. }
 The infimum 
{in the formula of $c_p$ is achieved} either at $0$ or $-1/2$ and hence 
$c_p=\min\{2^{1-p},1\}$. 
Using this twice we obtain 
\begin{equation}
\begin{split}\label{leavelower2}
|\omega_v-\omega_w|^p &=|s_v-s_w +(\omega^0_v-\omega^0_w)|^p \cr
&\geq c_p |s_v+ (\omega^0_v-\omega^0_w)|^p -|s_w|^p \cr
&\geq c^2_p |s_v|^p- c_p |\omega^0_v-\omega^0_w|^p -|s_w|^p. \cr
\end{split}
\end{equation}
We have 
\begin{equation}
\begin{split}
\label{XSppSOS}
	&\sum_{v \in \gamma, 
	w\leftarrow v}\Bigl( c_p^2 |s_v|^p-|s_w|^p \Bigr)\cr
&= (d+1)c_p^2 |s_0|^p+ (d c_p^2-1)
\sum_{v\in \gamma\backslash 0}|s_v|^p \geq (d c_p^2 -1)
\sum_{v\in \gamma}|s_v|^p.
\end{split}
\end{equation}
From this the claim follows. 
\endproof

 \subsection{{Stable inhomogeneous ground states for finite-state 
 models}}

 In the case of finite-state ferromagnetic models defined in \eqref{model-finite-spin}, the boundedness of increments is automatic, and the notion of stability becomes the following.

 {
\begin{defn}
  In our finite-state cases with potentials \eqref{model-finite-spin} we say that a configuration $\omega^0\in\Omega$ is stable with stability constant $c>0$ 
if for all $\omega\in\Omega$ which differ at most at finitely many sites from $\omega^0$, we have the lower bound 
\begin{equation}
\begin{split}
\label{XSpSOS3}
	&H(\omega)- H(\omega^{0}) \geq c\sum_{v}1_{\omega_v\neq\omega^0_v}.
\end{split}
\end{equation}
\end{defn}
}

We then have the following analogue of Lemma \ref{Lem1}.

\begin{lemma}\label{excess2} Consider the finite-state models \eqref{model-finite-spin}. 
	Let $\bar\gamma=(\omega_{\gamma},\gamma)$ 
	be a contour relative to the fixed ground state 
	$\omega^0$. 
	Denote $\omega=(\omega_{\gamma}\omega^0_{\gamma^c})$ 
	the corresponding excited spin configuration. 
	
	Then the excess energy satisfies
\begin{equation}
\begin{split}	
\label{XSfinite}
	&H(\omega)- H(\omega^{0}) 
	=\sum_{v \in \gamma, w\leftarrow v}\Bigl( \Phi(\omega_v,\omega_w)-
	\Phi(\omega^0_v,\omega^0_w)\Bigr)\cr
&	\geq (d-1)u |\gamma|-\sum_{v \in \gamma, w\leftarrow v}\Bigl(	\Phi(\omega^0_v,\omega^0_w)+u 1_{\omega^0_v\neq \omega^0_w}\Bigr) \cr
&\geq ((d-1)u-d_D( U+u)) |\gamma|.
\end{split}
\end{equation}
	\end{lemma}

\proof 
As $\gamma$ is the support 
of the contour we have $\omega_v\neq \omega^0_v$ 
for $v\in \gamma$ and $\omega_w= \omega^0_w$ for $w\not\in \gamma$. 

First, realize that
\begin{equation}
\begin{split}\label{leavelower3}
\Phi(\omega_v,\omega_w)\geq u 1_{\omega_v\neq \omega^0_v}
-u 1_{\omega_w\neq \omega^0_w}
- u 1_{\omega^0_v\neq \omega^0_w}.
\end{split}
\end{equation}
To see this write the inequality in the equivalent form 
\begin{equation}
\begin{split}\label{leavelower4}
\Phi(\omega_v,\omega_w)+u 1_{\omega^0_v\neq \omega^0_w}+ u 1_{\omega_v=\omega^0_v}
\geq u 1_{\omega_w=\omega^0_w}.
\end{split}
\end{equation}
The inequality trivially holds when the r.h.s. equals zero, so let us 
assume $\omega_w=\omega^0_w$. In the subcase
 $\omega_v=\omega^0_v$ the inequality obviously holds. 
 In the subcase $\omega_v\neq\omega^0_v$ it is impossible that the two first 
 terms on the l.h.s. reach zero at the same time. This proves the claim (\ref{leavelower3}).

Hence we have 
\begin{equation}
\begin{split}
\label{XSSOS}
	&\sum_{v \in \gamma, 
	w\leftarrow v}\Bigl(  
	1_{\omega_v\neq \omega^0_v}
-1_{\omega_w\neq \omega^0_w}
\Bigr)\cr
&= (d+1)1_{\omega_0\neq \omega^0_0}+ (d-1)
\sum_{v\in \gamma\backslash 0}1_{\omega_v\neq \omega^0_v} \geq (d-1)|\gamma|.
\end{split}
\end{equation}
From the last two inequalities the claim (\ref{XSfinite}) follows. 
\endproof

{
\begin{cor}\label{cor-finite-spin} For each interaction constants $u,U$, and each maximal {internal} degree 
  ${d_{\max}}\in \N_0$ there is a minimal degree 
  $d(u,U,{d_{\max}})$ such that for all $d\geq d(u,U,{d_{\max}})$, 
  all configurations {$\omega^0\in \G^0_q({d_{\max},d})$} are stable ground states.
  \end{cor}
}
  
  {
  \begin{proof}[Proof of Corollary \ref{cor-finite-spin}] {Take} 
  $$d(u,U,{d_{\max}})= 2+ \lfloor d_{\max}(u+U)/u\rfloor$$ Then indeed 
  $$(d-1)u-d_{\max}(u+U)=:c>0 $$ which by Lemma \ref{excess2} implies stability with the constant $c$. Note that $d(u,U,{d_{\max}})$ does not depend on $q$.
  \end{proof}
  }


\section{Properties of low-temperature states}\label{sec-properties}

Low temperature expansions on trees have unusual properties, 
as compared to similar expansions on lattices. 

First, there is the lack of limiting free energies, which is another way of 
saying that boundary terms are not smaller than volume terms. In particular, this provocates the failure of variational principles for Gibbs measures on trees (see e.g. \cite{BuPf95}, also \cite{FollSne77}, {Remarks 3.11},
for a valid "inner" variational principle). 
Next, on trees 
the complement of a support of a contours is never a connected set, 
and there are never interior connected components. This facilitates the extension of Peierls argument to {not} necessarily symmetric frameworks, as we discuss in Section \ref{sec-tightness}.

All of this requires care in proper handling when it comes to more 
subtle properties, see e.g. the decorrelation property \eqref{decorr} for unbounded support sets, which we use to prove extremality, after having properly introduced specific cutsets to take care of possibly {atypical tail-events of unbounded support}. 
We thus need to be precise to ensure convergence 
in particular in the case of unbounded spin models, 
and especially as we do not have homogeneity of our ground states.

\subsection{{p}-SOS models: {Proof of Theorem \ref{Theorem1}}}

\subsubsection{Convergence proof for the partition function}

We now turn to the proof of convergence of cluster expansion 
in the case of the $p$-SOS model, 
assuming, for $\eta>0$ sufficiently large, the lower bound of the form 
\begin{equation}
\begin{split}
\label{XSSOS4}
	&\beta(H(\omega)- H(\omega^{0})) 	\geq \eta\sum_{v\in \gamma}|\omega_v-\omega^0_v|^p
\end{split}
\end{equation}
{with $\bar\gamma=(\omega_{\gamma},\gamma)$ a contour relative to the fixed ground state $\omega^0$ and $\omega=(\omega_{\gamma}\omega^0_{\gamma^c})$ the corresponding excited spin configuration. 
This bound is given by the excess energy Lemma \ref{Lem1}.}

\smallskip

We start with a polymer partition function representation of the spin partition function 
in a finite volume $\Lambda$ with boundary condition equal to $\omega^{0}$, which reads 
\begin{equation}
\begin{split}
Z_{\Lambda}^{\omega^0}={\sum_{n\in\N}}\sum_{\bar \gamma_1, \dots, \bar \gamma_n}
\prod_{i=1}^n \rho(\bar\gamma_i)
\end{split}
\end{equation}
where the sum is over pairwise compatible polymers $\bar\gamma$ with activities 
$$
\rho(\bar\gamma)=e^{- \beta(H_{\gamma\cup \partial \gamma}(\omega_{\gamma\cup \partial \gamma})- 
H_{\gamma\cup \partial \gamma}(\omega^{0}_{\gamma\cup \partial \gamma}))}
$$
given in terms of the excess energy. 

\smallskip

In our case the pairwise compatibility relation $\bar\gamma_1 \sim \bar\gamma_2$ 
is equivalent to the separation of their supports, i.e. $\text{dist}(\gamma_1,\gamma_2)\geq 2$. 
Recall that, by definition, the spin configuration on the complement of 
the union of the supports 
of the polymers $\cup_{i=1}^n \gamma_i$ 
necessarily coincides with the ground state $\omega^{0}$. 
The aim of the cluster expansion is to write 
\begin{equation}
\begin{split}
\log Z_{\Lambda}^{\omega^0}= {\sum_{I}w_I}
\end{split}
\end{equation}
as an analytic function in the complex variables $\rho(\bar\gamma)$ for 
$\bar\gamma\in P_{\Lambda}$ where 
$P_{\Lambda}$ denotes the set of polymers in the finite volume $\Lambda$
for the given fixed ground state, and 
$(w_I)_{I\in \mathbb N_0^{P_{\Lambda}}}$ are the expansion 
terms. For fixed multi{-}index $I$, $w_I$ is proportional to $$\prod_{\bar\gamma \in P_{\Lambda}}\rho(\bar \gamma)^{I(\bar\gamma)}.$$ 
We have the following quantitative convergence criterion. 

\prop{For each degree $d\geq 2$ and interaction exponent 
$p\in (0,\infty)$ there is a finite constant $\eta_0(d,p)$ such that 
for all $\eta \geq \eta_0(d,p)$ 
the cluster expansion converges for all polymer activities in the polydisk 
\begin{equation}\label{invab}
\begin{split}|\rho(\bar \gamma)|\leq
\exp\left(-\eta\sum_{v\in \gamma}|\omega_v-\omega^0_v|^p\right) 
\end{split}
\end{equation}
for all $\bar \gamma$.

}\label{Prop3}

\rk{Note that we have an infinite polymer family even in finite volume 
as {there are infinitely many possible height configurations or contours $\bar\gamma$ 
with the same geometric support $\gamma$}.  }

\proof
We use thus the convergence criterion for the logarithm of the partition function of  
abstract polymer models of Bovier-Zahradn\'ik (\cite{BZ00}, Theorem 1, page 768).
%
We choose the generalized volume function $a$ which depends also on the heights on the contour support 
given by 
\begin{equation}
\begin{split}
\label{XSSOS5}
	&a(\bar\gamma)=A\sum_{v\in \gamma}|\omega_v-\omega^0_v|^p
\end{split}
\end{equation}
where $A>0$ can be chosen to our convenience, see below.

We are guaranteed of the convergence of the cluster expansion if 
the following two conditions hold, for suitable choices of $A>0$, $\delta\in (0,1)$ which 
will be made below.  

\smallskip 

\noindent\textit{Condition 1.} For any polymer $\bar \gamma$ 
\begin{equation}
\begin{split}\label{Condi1}
|\rho(\bar \gamma)| e^{a(\bar \gamma )}\leq \delta.
\end{split}
\end{equation}
\textit{Condition 2.} For any polymer $\bar \theta$ 
\begin{equation}
\begin{split}\label{Condi2}
\sum_{\bar \gamma \not\sim \bar \theta }
|\rho(\bar \gamma)|e^{a(\bar \gamma)}\leq \frac{1}{L(\delta)}a(\bar \theta)
\end{split}
\end{equation}
with $L(\delta)=-\log(1-\delta){/\delta}$. 
Here the sum is over \textit{incompatible} polymers $\bar \gamma \not\sim \bar \theta$. 
In our case incompatibility means that $\gamma$ and $\theta$ have graph 
distance less or equal than  $1$.

Under these two conditions, Theorem 1 of Bovier-Zahradn\'ik \cite{BZ00}
provides moreover the quantitative estimate 
\begin{equation}\label{66}
\begin{split}
\sum_{I\ni \bar\theta }|w_I|&\leq L(\delta)|\rho(\bar\theta)|e^{a(\bar \theta)}\cr
&\leq L(\delta)\exp\bigl(-(\eta-A)\sum_{v\in {\theta}}|\omega_v-\omega^0_v|^p\bigr)\cr
\end{split}
\end{equation}
where the sum $I$ is over those multi-indices $I$ which carry at least power 
one for polymer $\bar\theta$, and $w_I$ are the corresponding terms in the expansion 
of the logarithm of the partition function.  

{To treat our infinite family of contours with the Bovier-Zahradn\'ik convergence 
criteria for models of finite polymer families \cite{BZ00}, we may 
first use truncation of the height variables in modulus.    
We then treat the truncated models uniformly at fixed truncation, 
see in particular the uniformity of the relevant estimate \eqref{66} below in the truncation. 
In the final step one uses dominated convergence to remove the truncation. }

Now, the first condition \eqref{Condi1} is satisfied for $e^{A}\leq \delta e^{\eta}$. 

To treat the second condition \eqref{Condi2} is the more serious requirement.
We bound the sum over incompatible polymers as
\begin{equation}\begin{split}
&\sum_{\bar \gamma \not\sim \bar \theta }
|\rho(\bar \gamma)| e^{a(\bar \gamma)}
\leq \sum_{v \in \theta \cup \partial \theta}\sum_{\bar \gamma, \gamma \ni v}|\rho(\bar \gamma)| e^{a(\bar \gamma)}\cr
&\leq \sum_{v \in \theta \cup \partial \theta}
\sum_{\bar \gamma, \gamma \ni v}\exp\bigl( -(\eta-A)\sum_{w\in \gamma}|\omega_w-\omega^0_w|^p
\bigr)\cr
\end{split}
\end{equation}
Note that, while the contour activities on the l.h.s are in general not tree-automorphism 
invariant, we have used the {\it invariant bounds} \eqref{invab} for them to obtain 
the expression on the r.h.s. {The bounds are invariant as they suppress 
the activity in terms of the local deviation from the ground state  
$|\omega_w-\omega^0_w|^p$ with the same site-independent prefactor.}
To bound the last sum from above, we decompose it into 
a sum over the geometric parts of the contour with fixed volume $|\gamma|=l$ 
and, conditional on that, 
the sum over the height-configurations $\omega_v\neq \omega^0_v$.  
This provides 
\begin{equation}
\begin{split}
&\sum_{\bar \gamma, \gamma \ni v}
\exp\bigl( -(\eta-A)\sum_{w\in \gamma}|\omega_w-\omega^0_w|^p
\bigr)\cr
&=\sum_{l=1}^\infty\Bigl(\sum_{s\in \Z \backslash 0} 
e^{-|s|^p (\eta - A)})\Bigr)^l\,\, \#\{\gamma:\gamma \ni 0, |\gamma|=l.
\}
\end{split}
\end{equation}

The geometric entropy bound of 
Lemma 6 of \cite{GRS12} (formulated there for the Ising model, but valid in general as it is a purely combinatoric statement) 
provides the estimate 
\begin{equation}
\begin{split}
 \#\{\gamma:\gamma \ni 0, |\gamma|=l
\}\leq (d+1)^{2 (l-1)}
\end{split}
\end{equation}
in terms of the number of bonds $l-1$ of a subtree with $l$ vertices. 
Using the bound $|\theta \cup \partial \theta|\leq (d+2)|\theta|$ with equality 
for a singleton $\theta$,  we see that 
the second condition \eqref{Condi2} is implied if we have 
\begin{equation}
\begin{split}
(d+2)
\sum_{l=1}^\infty\Bigl(\sum_{s\in \Z \backslash 0} 
e^{-|s|^p (\eta - A)})\Bigr)^l\,\, (d+1)^{2 (l-1)}\leq \frac{A}{L(\delta)}.
\end{split}
\end{equation}
This is equivalent to 
\begin{equation}
\begin{split}
\sum_{l=1}^\infty\Bigl((d+1)^{2}\sum_{s\in \Z \backslash 0} 
e^{-|s|^p (\eta - A)})\Bigr)^l\,\, \leq \frac{(d+1)^2}{(d+2)}\frac{A}{L(\delta)}.
\end{split}
\end{equation}
Let us fix $A=1$ and $\delta = \frac{1}{2}$. 
Then we see that we may chose indeed  $\eta_0(d,p)<\infty$ such that 
for all $\eta \geq \eta_0(d,p)$  the inner 
series with summation over $s$ 
on the l.h.s. becomes small enough to ensure the validity of 
the desired inequality. Enlarging  $\eta_0(d,p)$ if necessary we can achieve 
that also the first condition holds, and this finishes the proof of Proposition \ref{Prop3}. 

\endproof

We now prove the existence of the infinite volume measure. Although getting it from convergent cluster expansions can be considered to be familiar on the lattice, the use of low temperature expansions on the tree is less standard 
and requires a careful treatment we present now.

\subsubsection{Finite-dimensional convergence and tightness}\label{sec-finite-dim-conv}

{
In the convergence regime of Proposition \ref{Prop3} for the expansion around 
stable $\omega^0$, we also have the associated finite-dimensional convergence 
of the finite-volume Gibbs measures. }

\begin{lemma}\label{Lem3} Let $W$ denote a finite subvolume of the vertex set of the infinite tree $V$. 
Then 

\begin{equation}
\begin{split}
\lim_{\Lambda \uparrow V}\mu_{\Lambda}^{\omega^0}(\sigma_W- \omega_W^0 \in \cdot) 
\end{split}
\end{equation}
exists as a weak limit. 
\end{lemma}

\proof To prove Lemma \ref{Lem3} we will show pointwise convergence of the Fourier-transform, and tightness of the measures. 

\medskip

{\bf Fourier transform. } Consider the Fourier transform for $t \in \R^W$ in finite volume $\Lambda$ 
\begin{equation}
\begin{split}
&\mu_{\Lambda}^{\omega^0}(e^{i\langle\sigma- \omega^0,t\rangle_W}) \cr
\end{split}
\end{equation}
with the notation $\langle a,b\rangle_W=\sum_{v\in W}a_v b_v$. 
Define $t$-dependent complex activities
$$\rho_t(\bar\gamma)=\rho(\bar\gamma)
e^{i\langle\sigma- \omega^0,t\rangle_{W\cap \gamma}}.
$$
We have $|\rho_t(\bar\gamma)|=\rho(\bar\gamma)$. 
Then 
\begin{equation}
\begin{split}
&\mu_{\Lambda}^{\omega^0}(e^{i\langle\sigma- \omega^0,t\rangle_W}) 
=\frac{{\sum_n}\sum^{\Lambda}_{\bar\gamma_1,\dots, \bar\gamma_n}\prod_{i=1}^n \rho_t(\bar\gamma_i)}{
{\sum_n}\sum^{\Lambda}_{\bar\gamma_1,\dots, \bar\gamma_n}\prod_{i=1}^n \rho(\bar\gamma_i)}\cr
&=\exp\Bigl({\sum^{\Lambda}_{I\cap W\neq \emptyset}}(w^t_I-w_I)\Bigr) 
\end{split}
\end{equation}
where $w_t$ are the cluster weights corresponding to the $t$-dependent activities, 
and the upper index $\Lambda$ in our notation $\sum^\Lambda$ indicates that the sums correspond to 
the family of polymers with supports inside the finite volume $\Lambda$.

Then, {by the Bovier-Zahradn\'ik bound \eqref{66}} there is absolute convergence of the cluster sums:
\begin{align*}
	\sup_{\Lambda}\sum^{{\Lambda}}_{I\cap W\neq \emptyset}(|w^t_I|+|w_I|) \leq 
\sum_{v\in W}\sum_{I\ni v}(|w^t_I|+|w_I|)\\
{\leq |W|\cdot 2L(\delta)\cdot \sum_{s\in\Z\backslash 0}e^{(-\eta+A)|s|^p}}\leq  |W| C
\end{align*}
and therefore the limit $\Lambda \uparrow V$ exists, pointwise in any $t \in \R^W$. 

\medskip

{\bf Quantitative {t}ightness estimate via contour-estimate.} \label{sec-tightness}
We start with the following upper bound. 
For any deviation $s\neq 0$ we have for the finite-volume fixed-site marginal 
\begin{equation}
\begin{split}
&\mu^{\omega^0}_{\Lambda}(\sigma_v=\omega_v^0+s)
=\sum_{\bar\gamma:\gamma\ni v, {\sigma_v=\omega_v^0+s}}^{\Lambda}\rho(\bar\gamma) \cdot
\frac{{\sum_n}\sum^{\Lambda}_{\bar\gamma_1,\dots, \bar\gamma_n \sim \bar\gamma}\prod_{i=1}^n \rho(\bar\gamma_i)}{
{\sum_n}\sum^{\Lambda}_{\bar\gamma_1,\dots, \bar\gamma_n}\prod_{i=1}^n \rho(\bar\gamma_i)}.\cr
\end{split}
\end{equation}

Now, on trees, as already underlined, contours  {have} 
no interior and {so we extend the Peierls argument to our non-symmetric cases by bounding  
 the fraction of polymer partition functions above by $1$.} We arrive thus at the upper bound 
\begin{equation}
\begin{split}
&\mu^{\omega^0}_{\Lambda}(\sigma_v=\omega_v^0+s)\leq \sum_{\bar\gamma:\gamma\ni v, {\sigma_v=\omega_v^0+s}}\rho(\bar\gamma)\cr
\end{split}
\end{equation}
where we extended the sum in $\Lambda$ to involve all polymers in $V$. 
The last sum is bounded by 
\begin{equation}
\begin{split}
&e^{-|s|^p \eta } \sum_{l=1}^\infty\#\{ \gamma\ni v: |\gamma|=l\}\Bigl(\sum_{r\in \Z \backslash 0} 
e^{-|r|^p\eta}
\Bigr)^{l-1} \cr
&\leq e^{-|s|^p \eta }C(\eta,d,p) \cr
\end{split}
\end{equation}
where 
\begin{equation}
\begin{split}\label{Ceta}
&C(\eta,d,p)
:= \sum_{m=0}^\infty
\Bigl((d+1)^{2} \sum_{r\in \Z \backslash 0} 
e^{-|r|^p \eta }\Bigr)^m\downarrow 1\cr
\end{split}
\end{equation}
as $\eta\uparrow \infty$ for fixed $d,p$, and is in particular finite for $\eta$ sufficiently large.

This provides tightness of the single-site marginals of the family $\mu^{\omega^0}_{\Lambda }$ 
as one gets a uniform estimate in $\Lambda$ of the form 
\begin{equation}
\begin{split}
&\mu^{\omega^0}_{\Lambda}(|\sigma_v-\omega_v^0|\geq N)
\leq  2 C(\eta,d,p) \sum_{s\geq N}e^{-|s|^p \eta }.
\end{split}
\end{equation}
Consequently we get for the marginals in finite volume $W$ 
\begin{equation}
\begin{split}
&\mu^{\omega^0}_{\Lambda}(\max_{v\in W}|\sigma_v-\omega_v^0|\geq N)
\leq  2 | W | C(\eta,d,p) \sum_{s\geq N}e^{-|s|^p \eta }.
\end{split}
\end{equation}
This is the desired tightness of the finite-volume marginals 
\begin{equation}
\begin{split}
&\lim_{N\uparrow \infty}\sup_{\Lambda}\mu^{\omega^0}_{\Lambda}(\max_{v\in W}|\sigma_v-\omega_v^0|\geq N)=0.
\end{split}
\end{equation}
Using the L\'evy-continuity theorem we obtain the existence of the pointwise limit 
of the characteristic functions and the tightness the weak convergence of the finite-dimensional 
marginals of $\mu_{\Lambda}^{\omega^0}$.

Alternatively we could have concluded the convergence of the marginals on $W$ without tightness 
from the L\'evy continuity theorem by proving instead 
continuity of the limit of the Fourier transform in $t=0$ which can be seen 
from the uniform convergence of the cluster-sums. 

{This concludes the proof of Lemma \ref{Lem3}. \endproof}

\subsubsection{DLR-property and identifiability of family of measures}\label{sec-DLR}

Note that the limiting finite-dimensional marginals indexed by $W$ of Lemma \ref{Lem3}
provide a family of consistent measures in the sense of Kolmogorov's extension theorem, 
and hence define an infinite-volume measure $\mu^{\omega^0}$,  for each of 
the corresponding sparse ground states $\omega^0$. Let us see that {$\mu^{\omega^0}$}
is also a Gibbs measure,  in the usual DLR sense \cite{HOG}, using again the convergence of 
finite-dimensional marginals, and consistency of the kernels.

\medskip

{\bf DLR-property of limiting measures.} 
Note first that for any cofinal\footnote{{See \cite{HOG}. A subset $\mathcal{S}_0$ of an index set $\mathcal{S}$ directed by inclusion is called {cofinal} if each $\Lambda \in \mathcal{S}$ is contained in some $\Delta \in \mathcal{S}_0$.}} sequence $(\Lambda_n)_n$, any  $\omega_\Lambda\in \Omega_\Lambda$ for $\Lambda \Subset V$ we have for $n$ sufficiently large 
that 
\begin{equation}
\begin{split}
&\int\mu^{\omega^0}_{\Lambda_n}(d\tilde\omega_{\partial \Lambda})\gamma_{\Lambda}
(\sigma_{\Lambda}=\omega_{\Lambda}| \tilde\omega_{\partial \Lambda})=\mu^{\omega^0}_{\Lambda_n}(\sigma_{\Lambda}=\omega_{\Lambda}).\cr
\end{split}
\end{equation}
{But from the convergence of finite-dimensional marginals,  the spatial Markov property 
of the kernel follows, and after performing the large $n$-limit, 
\begin{equation}
\begin{split}
&\int\mu^{\omega^0}(d\tilde\omega_{\partial \Lambda})\gamma_{\Lambda}
(\sigma_{\Lambda}=\omega_{\Lambda}| \tilde\omega_{\partial \Lambda})
=\mu^{\omega^0}(\sigma_{\Lambda}=\omega_{\Lambda})\cr
\end{split}
\end{equation}
which is the DLR-equation (see e.g. \cite{HOG}). }

\medskip

{\bf Identifiability of states from the set of sparse {ground states.}} 
We show now that we have $\mu^{\omega^0}\neq \mu^{\tau^0}$ when $\omega^0\neq \tau^0$ 
and both are from the set of stable ground states obeying \eqref{generalsparsity}. 
To see this consider $v$ such that $\omega^0_v\neq \tau^0_v$ 
and compute 
\begin{equation}
\begin{split}
&\mu^{\omega^0}(\sigma_v=\omega_v^0)\geq 1-2 C(\eta,d,p) \sum_{s\geq 1}e^{-|s|^p \eta }
\end{split}
\end{equation}
(which is close to $1$ by \eqref{Ceta}). Compare to 
\begin{equation}
\begin{split}
&\mu^{\tau^0}(\sigma_v=\omega_v^0)\leq C(\eta,d,p) e^{-|\tau^0_v-\omega^0_v |^p \eta}
\end{split}
\end{equation}
which is close to zero, and hence both probabilities are different for $\eta$ sufficiently large: one can identify different states for different low-temperatures excitations of different sparse ground states. 

\subsubsection{Extremality via cutsets and decorrelation}\label{sec-cutset}

{\bf Cutset property.}} 
Fix a ground state $\omega^0$. Let $U\subset W$ be two finite nested volumes. 
We say that there is {\it an $(U,W)$-cutset (w.r.t $\omega^0$) in the configuration $\omega$} 
if every path from $U$ to infinity has a site $v\in W\setminus U$ for which $\omega_v=\omega^0_v$. 

First we prove that there are always cutsets around arbitrary large volumes in sufficiently 
large annuli, with $\mu^{\omega^0}$-probability arbitrarily close to one: 

\begin{lemma} \label{Lem4} Let $v$ be an arbitrary vertex and $B(r,v)\subset V$ 
the ball of radius $r$ and center $v$ on the tree w.r.t. to the graph distance. 

Then, for any radius $r$ and $\epsilon >0$ there exists a finite radius $R>r$ such 
there is a $\big(B(r,v),B(R,v)\big)$-cutset with $\mu^{\omega^0}$-probability at least $1-\epsilon$. 
\end{lemma}

\proof  Define {$N_v$} to be the $(\N_0\cup {\{\infty\}})$-valued random variable which 
gives the size of a contour {containing $v$}. Clearly, in finite volume with boundary 
condition $\omega^0$, its value is bounded by the size of the finite volume.  
Let us now derive an exponential bound on the tail 
which is uniform in the volume, and also holds in infinite volume. 

\smallskip

To do so, look at the exponential moment generating function, which we will do 
first in finite volume $\Lambda$. Rewrite in terms of polymer partition functions 
for $t>0$ and estimate   
\begin{equation}
\begin{split}
&\mu^{\omega^0}_{\Lambda}(e^{t N_v})= \mu_{\Lambda}^{\omega^0}(N_v=0)
+\sum_{\bar\gamma:\gamma\ni v}^{\Lambda}\rho(\bar\gamma)e^{t |\gamma|}
\frac{\sum_n\sum^{\Lambda}_{\bar\gamma_1,\dots, \bar\gamma_n \sim \bar\gamma}\prod_{i=1}^n \rho(\bar\gamma_i)}
{\sum_n\sum^{\Lambda}_{\bar\gamma_1,\dots, \bar\gamma_n}\prod_{i=1}^n \rho(\bar\gamma_i)}\cr
&\leq 1 + \sum_{\bar\gamma:\gamma\ni v}^{\Lambda}\rho(\bar\gamma)e^{t |\gamma|}.\cr
\end{split}
\end{equation}
The last line follows, as all the terms inside 
the polymer partition function in the numerator are contained in the polymer partition function 
in the denominator. 
Decomposing the contour sum in the last line over contours of size $l$ and using the entropy 
estimate as above, the r.h.s. of the last display is bounded above by 
\begin{equation}
\begin{split}
1+ \sum_{l=1}^\infty
(d+1)^{2(l-1)}\Bigl( \sum_{r\in \Z \backslash 0} 
e^{-|r|^p \eta}\Bigr)^{l-1} e^{t l}=:L(t).\cr
\end{split}
\end{equation}
$L(t_0)$ is clearly finite for any choice of $t_0$ such that the geometric $l$-sum converges, i.e. s.t.
\begin{equation}
\begin{split}
\sum_{l=1}^\infty
(d+1)^{2}\sum_{r\in \Z \backslash 0} 
e^{-|r|^p \eta} e^{t_0}<1.\cr
\end{split}
\end{equation}
Assuming such a choice for $t_0$, {we deduce by the Markov inequality the uniform exponential upper bound on the size-distribution
of the contour containing $v$}:
\begin{equation}
\begin{split}
&\mu^{\omega^0}_{\Lambda}(N_v\geq N)\leq L(t_0) e^{- N t_0}
\end{split}
\end{equation}
which extends also to the infinite-volume measure. 

We may use this exponential bound on the contour size distribution to control  
the non-existence event of a cutset in an annulus. 
To see this, consider contours anchored at the boundary of the inner volume  $B(r,v)$ 
and note  
\begin{equation}
\begin{split}
&\mu^{\omega^0}(\text{there is no cutset in }B(R,v)\backslash B(r,v))\cr 
&\leq \mu^{\omega^0}(\text{there is site  }w \in \partial B(r,v):  N_w\geq R-r )\cr
&\leq |\partial B(r,v)| L(t_0) e^{- (R-r) t_0}
\end{split}
\end{equation}
This can be made smaller than $\epsilon$ by chosing $R$ large enough, 
{which proves Lemma \ref{Lem4}. \endproof }

\medskip

{\bf Correlation decay for events of polymer type.} \label{sec-correlation-decay}
We start with correlation bounds for events which can be nicely expressed in terms of contours 
and polymer partition functions. 
We say that a local event $A\in \F_{W}$ is of \textit{polymer-type with supporting set $W$}, 
if $\omega\in A$ implies that 
$\omega_v=\omega^0_v$ for all sites $v$ in the inner boundary of $W$. 
Of course not every local event is of such a type, for example the event 
$\{\omega_v=\omega^0_v+1\}$ is not of this form. 

\begin{lemma} (Decay of polymer correlations.)   \label{Lem5}
For any local events $A\in \F_{W}$ and $B\in \F_{U}$ of polymer-type
we have 
$$|\mu^{\omega^0}(A\cap B)-\mu^{\omega^0}(A) \mu^{\omega^0}(B)|\leq 
\phi\bigl(|W|,d(W,U)\bigr) $$ 
with a decay function of the form 
$\phi(|W|,r)\leq \exp(C |W|e^{- c r})-1$, for $C,c >0$. 
\end{lemma}

{\bf Remark.} Note that the estimate is uniform in the size of one of the volumes, 
which is here chosen to be $|U|$. 

\proof The proof of the lemma follows by expressing the events $A,B$ in question as unions over 
events formulated in terms of two finite polymer families, one  of the families 
with supports inside of $W$, the other one in  $U$ respectively. 

\smallskip

Let us write for short $\mu$ for the infinite-volume measure {$\mu^{\omega_0}$}.
We have for the probability of a single contour $a$  the expression 
$$\mu(a)=\rho(a)\exp (-\sum_{I\not\sim a}w_I).$$
The similar expression for a family of contours reads 
$$\mu(a_1, \dots, a_l)= \prod_{i=1}^l\rho(a_i)\exp (-\sum_{I: \exists i\in \{1,\dots, l\}:I\not\sim a_i}w_I).$$
This gives for the correlation between two families of contours {with disjoint supports}
\begin{equation}
\begin{split}
&\mu(a_1, \dots, a_l, b_1, \dots, b_m)- \mu(a_1, \dots, a_l)\mu( b_1, \dots, b_m)\cr
&=\prod_{i=1}^l\rho(a_i)\prod_{j=1}^l \rho(b_i)\Bigl(
\exp (-\sum_{I: \exists c\in \{a_1,\dots, a_l, b_1, \dots, b_m\}: I\not\sim c}w_I)\cr 
&-\exp (-\sum_{I: \exists i\in \{1,\dots, l\}:I\not\sim a_i}w_I)\exp (-\sum_{I: \exists j\in \{1,\dots, m\}:I\not\sim b_i}w_I)\Bigr)
\end{split}
\end{equation}
which can be rewritten as 
\begin{align}
&\mu(a_1, \dots, a_l)\mu( b_1, \dots, b_m)\Bigl(
\exp (-\sum_{I: \exists c\in \{a_1,\dots, a_l, b_1, \dots, b_m\}: I\not\sim c}w_I \nonumber\\
&\hspace{4cm}+\sum_{I: \exists i\in \{1,\dots, l\}:I\not\sim a_i}w_I 
+ \sum_{I: \exists j\in \{1,\dots, m\}:I\not\sim b_i}w_I)
-1\Bigr)\\
&{=\mu(a_1, \dots, a_l)\mu( b_1, \dots, b_m)\Bigl(
	\exp (+\sum_{I: \exists (i,j)\in \{1,\ldots,\ell\}\times\{1,\ldots,m\}: I\not\sim \{a_i,b_j\}}w_I )
	-1\Bigr)}
\end{align}
This means that only clusters $I$ are surviving 
 which connect the supporting sets $W$ and $U$, which are controlled 
 by the number of anchoring points in $|W|$ and the cluster expansion estimates. 
{By \eqref{66}, the argument of the above exponential term is indeed bounded by
\begin{align}
	&	\sum_{w\in W}\sum_{\substack{\gamma\ni w :\\ |\gamma|\geq d(W,U)}} \sum_{I\ni\gamma} w_I\\
	&\leq L(\delta)|W|\sum_{\ell=d(W,U)}^\infty 
		(d+1)^{2(\ell-1)}
		\left(\sum_{r\in\Z\backslash 0} e^{(-\eta+A)|r|^p}\right)^\ell\\
	&\leq  |W| C(\delta,d) e^{-c(d,\eta) d(W,U)}.
\end{align}
}

  This delivers the existence of the 
 decay function $\phi$ with the promised property such that 
\begin{equation}
\begin{split}
&\mu(a_1, \dots, a_l, b_1, \dots, b_m)- \mu(a_1, \dots, a_l)\mu( b_1, \dots, b_m)\cr
&\leq \mu(a_1, \dots, a_l)\mu( b_1, \dots, b_m)\phi(|W|,{d(W,U)}).
\end{split}
\end{equation}
Finally note that 
 the estimate survives the summation over all possible different families of 
 contours inside $W,U$ (which form the decomposition of the events $A,B$), 
 as the prefactors sum up at most to $1$. {Indeed, let $A=A_W\times\Omega_0^{W^c}\in\mathcal F_W$ and $B=B_U\times\Omega_0^{U^c}\in\mathcal F_U$, then
 \begin{align}
	|\mu(A\cap B)-\mu(A)\mu(B)|&=
	|\sum_{\omega_W\in A_W}\sum_{\eta_U\in B_U}\mu(\omega_W\eta_U)-\mu(\omega_W)\mu(\eta_U)|\\
	&\leq \sum_{\omega_W\in A_W}\sum_{\eta_U\in B_U} \mu(\omega_W)\mu(\eta_U)\phi(|W|,d(W,U))\\
	&\leq \phi(|W|,d(W,U)).
 \end{align}}
 { This proves Lemma \ref{Lem5}. }
 \endproof

{\bf Extremality via decorrelation of general events via cutsets.} 
We turn to the proof of {extremality of the above constructed measures $\mu^{\omega_0}$.} {By [Proposition 7.9] of \cite{HOG}}, it is equivalent to show the following

\prop{
For any fixed $A\in \F$, and cofinal volume sequence $(\Lambda_n)_{n\in \N}$ 
the decorrelation property holds
\begin{equation}
\begin{split}\label{decorr}
&\lim_{n\to\infty}\sup_{B\in \F_{\Lambda^c_n}}|\mu^{\omega^0}(A\cap B)-\mu^{\omega^0}(A)\mu^{\omega^0}(B)|=0.
\end{split}
\end{equation}

}\label{Prop4}

\rk{
Indeed, from \eqref{decorr} the tail-triviality of $\mu=\mu^{\omega^0}$ follows by taking
$A=B$ to be a tail-event. This is an allowed choice in the above  limit statement, 
which delivers the desired formula $\mu(A)=\mu(A)^2$. }

\proof
Spelling out the above general decorrelation property we are aiming at, means that for any {$A\in\mathcal F$ and $\delta>0$} there 
exists $n_0(A,\delta)\in\N$ such that for all $n\geq n_0$ we have for any $B\in \F_{\Lambda^c_n}$
\begin{equation}
\begin{split}
&|\mu(A\cap B)-\mu(A)\mu(B)|\leq \delta.
\end{split}
\end{equation}
{To see this, we apply the {\it semi{-}ring approximation theorem} twice, and condition on the presence of suitable cutsets, as follows. } 
Note first that by the semi{-}ring approximation theorem applied to the semi{-}ring 
of cylinder events, for any $\epsilon>0$ and any event $A {\in \mathcal F}$ we may 
choose a cylinder event $A_{\epsilon}$ such that $\mu(A\triangle A_{\epsilon} )\leq \epsilon$. 
See e.g. the book of Klenke \cite{Klenke14}, Theorem 1.65 (ii). 

It is now elementary that we can choose $\epsilon>0$ small enough such that for any four events for which 
$\mu(A\triangle A_{\epsilon} )\leq \epsilon$ and $\mu(B\triangle B_{\epsilon} )\leq \epsilon$
we always have 
\begin{equation}
\begin{split}
|\mu(A; B)-\mu(A_\epsilon; B_\epsilon)|\leq \delta/2
\end{split}
\end{equation}
where {$\mu(A; B):=\mu(A\cap B)-\mu(A)\mu(B)$}.
(The choice of $\epsilon=\frac{\delta}{{8}}$ will do for this\footnote{{Note that $(A\cap B)\Delta(A_\epsilon\cap B_\epsilon)\subset(A\Delta A_\epsilon)\cup(B\Delta B_\epsilon).$}}). 
Given $A \in\mathcal F$, let us fix the cylinder set $A_{\epsilon}$ which is obtained in such a way. 
It then  suffices to show that there exists 
$n_1=n_1(A_\epsilon,\delta)\in\N$ such that for all $n\geq n_1$ for any \textit{cylinders} 
$B_\epsilon\in \F_{\Lambda^c_n}$ 
we have 
\begin{equation}
\begin{split}
&|\mu(A_\epsilon; B_\epsilon)|\leq \delta/2.
\end{split}
\end{equation}
As the {approximating cylinder events} $A_{\epsilon},B_{\epsilon}$ may {\it not} be of 
polymer type as introduced above, we can not directly apply the decay estimate of Lemma \ref{Lem5}
for those events 
without further ado. We solve this problem by the introduction of cutsets, which 
occur with high probability, in the following way depicted in Figure \ref{figure}.

\medskip

\medskip
Fix a vertex $v$, and choose $r_1$ s.t. $A_{\epsilon}\in \F_{B(r_1,v)}$. 
For radii $r_1<R_1<r_2<R_2<r_3<R_3$,  
consider cutset events of the type
$$C_i:=\{\text{there is a cutset in } B(R_i,v)\backslash B(r_i,v) \} $$
for $i=1,2,3$. Let us now construct the radii. {Fix $\delta'>0$, to be chosen below.} 

\smallskip
{\bf Annulus for inner cutset.} {By Lemma \ref{Lem4},} choose $R_1>r_1$ large enough such that with probability at least $1-\delta'$ there is a cutset
in the annulus with $r_1,R_1$. 

\smallskip

{\bf Decorrelation annulus. } {Given $R_1$, by Lemma \ref{Lem5},} choose $r_2$ large enough such that 
$\phi(|B_{R_1}|, r_2-R_1)\leq \delta/4$. 

\smallskip

{\bf Annulus for middle cutset. }  Choose $R_2$ large enough such that with probability at least $1-\delta'$ there is a cutset
in the annulus with $r_2,R_2$. Choose $n$ large enough such that $\Lambda_n\supset B(R_2,v)$. 

\smallskip

{\bf Outmost cutset.  } Let $B_{\epsilon}$ be a cylinder set in $\F_{\Lambda^c_n}$. Choose $r_3$ such that 
$B_{\epsilon}\in \F_{B(r_3,v)}$. 
Choose $R_3$ large enough such that with probability at least $1-\delta'$ there is a cutset
in the annulus with respective radii $r_3$ and $R_3$.

\begin{figure}
    \begin{center}
    \includegraphics[width=12cm]{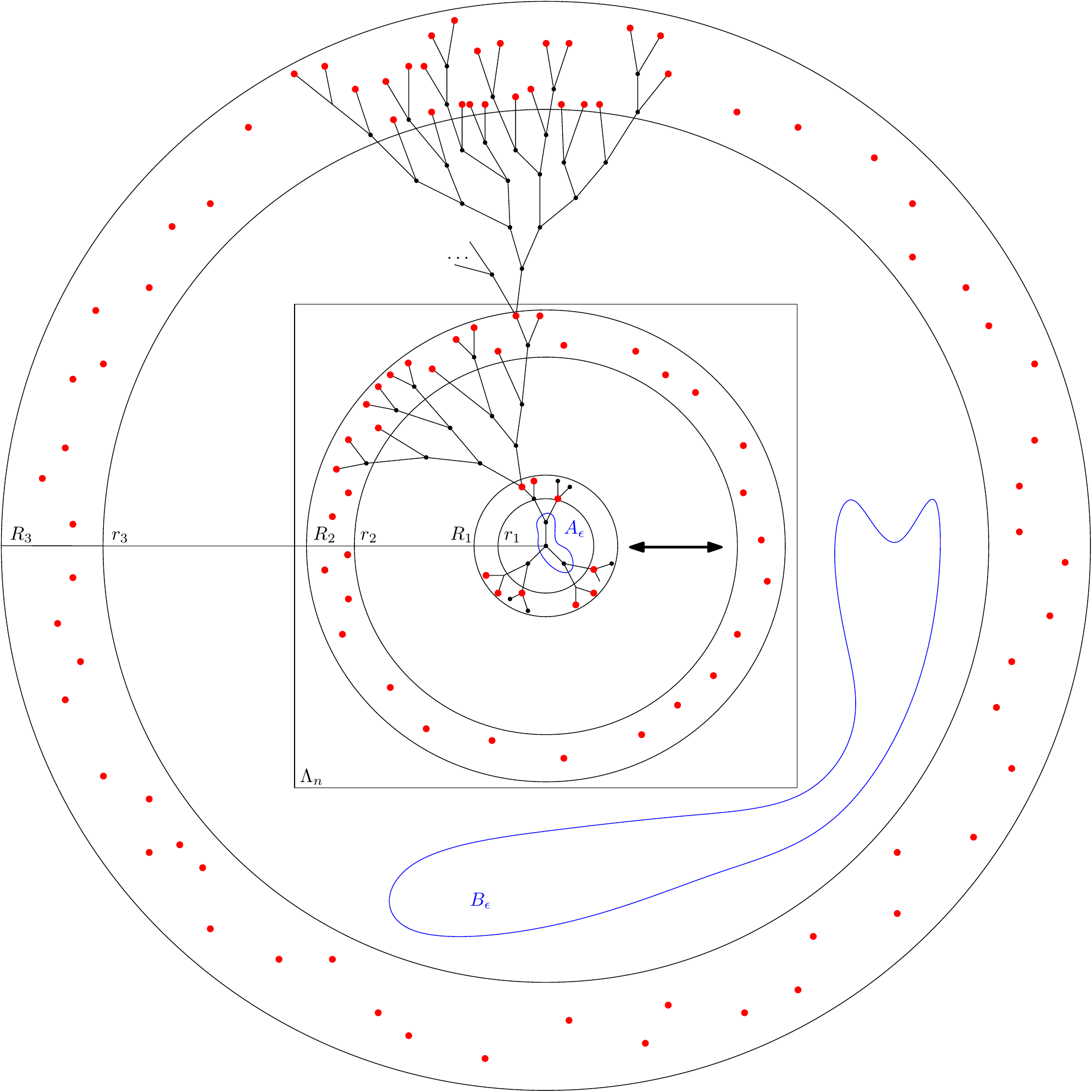}
    \caption{The picture shows the events $A_\epsilon'$ and $B_\epsilon'$ on the binary tree where $d=2$: in blue are the supports of the cylinder events $A_\epsilon$ (inside the ball of radius $r_1$) and $B_\epsilon$ (inside the ball of radius $r_3$ and outside the box $\Lambda_n$). The (larger) red dots represent  sites $v$ on the cutsets where $\sigma_v=\omega^0_v$. The decorrelation annulus (between radii $R_1$ and $r_2$) is emphasized with a bold arrow. }\label{figure}
    \end{center}
\end{figure}

Now define the following events, as depicted on Figure \ref{figure}:
$$A'_{\epsilon}:=A_{\epsilon}\cap C_1, \qquad B'_{\epsilon}:=B_{\epsilon}\cap C_2\cap C_3.$$
The advantage of these events is that they are of polymer-type with well-separated supporting sets, and by construction 
enjoy the decorrelation property
\begin{equation}
\begin{split}
&|\mu(A'_\epsilon;B'_\epsilon)|\leq \delta/4.
\end{split}
\end{equation}
Noting that $\mu(A'_{\epsilon}\triangle A_{\epsilon})\leq \delta'$ and $\mu(B'_{\epsilon}\triangle B_{\epsilon})\leq 2\delta'$ by the construction of the width of the radii for the cutsets, 
we now assume a choice of $\delta'$ has been made such that 
\begin{equation}
\begin{split}
&|\mu(A'_\epsilon; B'_\epsilon)-\mu(A_\epsilon; B_\epsilon)|\leq \delta/4.
\end{split}
\end{equation}
$\delta'=\frac{\delta}{48}$ will do for this. 
This finishes the proof of {Proposition \ref{Prop4}.}
\endproof

\subsection{Finite-spin models: {Proof of Theorem \ref{Theorem2}}}

We discuss the {corre{s}ponding} proofs for the finite-spin models {defined in \eqref{model-finite-spin}}. 
We are assuming the lower bound of the form 
\begin{equation}
\begin{split}
\label{XSSOS6}
	&\beta(H(\omega)- H(\omega^{0})) 	\geq \zeta |\gamma|
\end{split}
\end{equation}
with $\zeta>0$, which is provided by Lemma \ref{excess2} for the excess energy 
{with respect to} configurations $\omega^0$ which are elements in 
the set of stable {ground states}  $\G^0_q(d_D)$, which was defined in the statement 
of Theorem \ref{Theorem2}. 
 
The definition of {labelled} contours $\bar\gamma$ w.r.t.  {a fixed stable reference ground state}
$\omega^0$  stays the same, i.e. they 
are connected sets of incorrect points, together with the spin values 
on these sets.  
We are again using representations in terms of polymer partition functions in 
which the contours carry activities 
$$
\rho(\bar\gamma)=e^{- \beta(H_{\gamma\cup \partial \gamma}(\omega_{\gamma\cup \partial \gamma})- 
H_{\gamma\cup \partial \gamma}(\omega^{0}_{\gamma\cup \partial \gamma}))}
$$
given in terms of the excess energy. 

{We first prove a convergence criterion for the low-temperature 
expansions, which parallels Proposition \ref{Prop3}, 
but assumes only volume-suppression for contour-activities in the following form.} 

\prop{For each degree $d\geq 2$ and 
$q\in (2,\infty) $ there is a finite constant $\zeta_0(d,p)$ such that 
for all $\zeta \geq \zeta_0(d,p)$ 
the cluster expansion converges for all complex polymer activities in the polydisk 
\begin{equation}\label{invablabla}
\begin{split}|\rho(\bar \gamma)|\leq
\exp\left(-\zeta |\gamma|\right) 
\end{split}
\end{equation}
for all $\bar \gamma$.  
}\label{Prop5}

{{\bf Remark. } Note that the r.h.s. depends only 
on the volume $|\gamma|$ of the labelled contour $\bar\gamma$, as opposed to the configuration-dependent assumption 
in Proposition \ref{Prop3}. 
Such uniformity in the spin configuration on $\gamma$ is only possible for finite-spin models. } 

\proof
{We choose the generalized volume function $b$ for labelled contours 
in the application of Theorem 1 of Bovier-Zahradn\'ik (\cite{BZ00}, Theorem 1, page 768), 
only depending on the volume of the contour in the form 
\begin{equation}
\begin{split}
\label{XSSOS7}
	&b(\bar\gamma)=B |\gamma|
\end{split}
\end{equation}
where the prefactor $B>0$ can be chosen to our convenience, see below.  }

Rerunning the convergence proof for the partition function 
as before, we see the following. 
The first condition \eqref{Condi1} is satisfied for $e^{B}\leq \delta e^{\zeta}$, by the assumption \eqref{invablabla}.
The second condition  \eqref{Condi2} is now implied if we have 
\begin{equation}
\begin{split}\label{qappearance}
(d+2)\sum_{l=1}^\infty e^{-l (\zeta - B)}(q-1)^l (d+1)^{2 (l-1)}\leq \frac{B}{L(\delta)}.
\end{split}
\end{equation}
Note that the number of choices of spin-values {per site} on the contour 
support is $q-1$, which is responsible for its appearance on the l.h.s. of \eqref{qappearance} above. 
We may finally choose $B=1, \delta=\frac{1}{2}$ to see that we satisfy both 
conditions by choosing $\zeta$ large enough. This proves 
Proposition \ref{Prop5}. \endproof 
\smallskip
{Having seen this, 
the remaining parts of the proof of the main theorem all  carry over from the {$p$-}SOS analogues. 
This including convergence via Fourier transform (while tightness is automatic),  
polymer decorrelation, and finally the extremality of the measures $\mu^{\omega^0}$ via
the general correlation decay of Proposition \ref{Prop4}, which holds 
by means of the Lemma \ref{Lem4} on cutsets.   
We finally remark that from (\ref{qappearance}) follows
that for the $q$-state Potts model we have a bound on the 
minimal inverse temperature $\beta$ 
for which all low-temperature states exist, 
on the order of $\log (q-1)$.   } 	

This provides the proof of Theorem \ref{Theorem2}. \endproof 

{\bf Remark.} We are not after optimality of the degree $d$ of the tree for which 
our states exist. 
It may be possible to extend the construction in the $p$-SOS and finite-spin cases 
to include also trees of low degree, even the binary tree, 
by demanding the distance between broken bonds in the ground state to be large. 
This has been outlined for the particular case of {foliated states} for the Ising model on binary trees in \cite{GRS15}, 
whether it can be done for our models we leave for future work.

%

\smallskip

\section{Applications to inhomogeneous systems with local disorder terms}
\label{Section5}

\subsection{Models}
Let us consider our previously discussed 
$\Z_q$-valued {or} $\Z$-valued models with homogeneous pair interactions $\Phi$, {defined in \eqref{model-p-sos} and \eqref{model-finite-spin}}, but under 
the additional influence in the interaction of single-site terms $\Psi_v:\Omega_0\rightarrow \R$ at the sites $v\in V$. So the Hamiltonian 
now takes the following form\begin{equation}
\begin{split}\label{additionalsinglesite}
\sum_{v\sim w}\Phi(\omega_v,\omega_w)+\sum_{v}\Psi_{v}(\omega_v).
\end{split}
\end{equation}
It is in general spatially inhomogeneous, but the case $\Psi_v=\Psi$ of 
a homogeneous local potential is not excluded, and already of interest.   
Let us highlight some prototypical special cases. 



\subsubsection*{\bf $\Z_q$-valued Potts model in quenched random potentials, random field 
Ising model}

As before we write the pair potential of the Potts model in the form $\Phi(\omega_v,\omega_w)=1_{\omega_v\neq \omega_w}$. 
Moreover, the single-site term $\Psi_{v}(\cdot)$ is a quenched 
random potential on $\Z_q$, which is usually assumed 
to be i.i.d. over the sites $v\in V$, according to an external probability distribution $\P$, 
and studied w.r.t. to its $\P$-a.s. properties.
The case of a deterministic single-site interaction where 
$\Psi_{v}(\cdot)=\Psi(\cdot)$ is allowed and models the Potts model in a homogeneous 
vector-valued field.

The subcase $q=2$ is identical to the random field Ising model (RFIM) for spins $\omega_v\in\{-1,1\}$, and quenched random fields $\eta_v\in \R$, 
with Hamiltonian 
$$
-\sum_{v\sim w}\omega_v \omega_w - \sum_{v}\eta_v \omega_v 
$$
which was already considered 
on the tree by Bleher et al. \cite{BleherRFIM}. 
These authors proved in particular that for the RFIM in the low temperature (large $\beta$) regime,
there is a strictly positive maximal strength $\eta^*(\beta)>0$ such that for all random field configurations 
$\eta=(\eta_v)_{v\in V}$ with $\sup_{v\in V}|\eta_v|\leq 
\eta^*(\beta)$ there are at least two different $\eta$-dependent Gibbs measures 
$\mu^+[\eta], \mu^-[\eta]$. These infinite-volume measures 
are obtained as weak limits of the finite-volume measures 
with all-plus (all-minus) boundary conditions.   
In their result the actual 
distribution under $\P$ plays no role. \\
We are aiming in this section at a broad generalization of this statement to extremal measures 
$\mu^{\omega^0}[\eta]$ constructed with non-homogeneous spin-boundary conditions 
$\omega^0$, and the more general model classes under discussion here.

\subsubsection*{\bf $\Z$-valued random field random surface models: random 
field $p$-SOS model}
In this variation of the $p$-SOS model the Hamiltonian takes the form 
\begin{equation}
\begin{split}\label{randomrandomsurface}
\sum_{v\sim w}|\omega_v-\omega_w|^p+\sum_{v}\eta_v \omega_v
\end{split}
\end{equation}
with quenched random fields $\eta_v\in \R$ and spins $\omega_v\in \Z$. 
Note that the local disorder term
of random-field type adds an unbounded perturbation to the Hamiltonian, even 
for uniformly bounded random fields $\eta_v$.
The model is well-defined for $p\geq 1$, while the 
case $p<1$ would lead to infinite partition functions for non-zero external fields.  
It has been recently studied in detail 
on the lattice for $p=2$ in 
\cite{DHP21} with a focus on the case where $(\eta_v)_{v\in V}$ are symmetrically distributed 
i.i.d. quenched random variables which have mean zero 
and finite variance. The authors obtained in their work 
upper and lower bounds on disorder-averages of the gradient fluctuations 
$|\omega_v-\omega_w|^2$ 
w.r.t. to the finite-volume zero-boundary condition Gibbs measure $\mu^0_{\Lambda}[\eta]$, 
valid for all finite boxes $\Lambda$. They provide boundedness of gradient 
fluctuations uniformly in the box-size in $d\geq 3$, and roughening of local fluctuations 
in $d\leq 2$.

Continuous-spin versions of the model with spin values $\omega_v\in \R$, also 
thereby allowing more general pair-interactions, were studied in \cite{DHP21,CK12}. 
We point out that the model \eqref{randomrandomsurface}
is of gradient-type, due to the multiplicative nature of the local terms, 
which opens the field for the study of {\it Gradient Gibbs measures } in the infinite volume 
which are defined on configurations of heights modulo a joint height shift, with state-space $\Z^V/\Z$.

\subsubsection*{\bf $\Z$-valued p-SOS models in random media}

Here one keeps the gradient interaction, but allows more general local interactions, so that the Hamiltonian 
becomes 
\begin{equation}
\begin{split}\label{SOSrandommedia}
\sum_{v\sim w}|\omega_v-\omega_w|^p+\sum_{v}\zeta_v(\omega_v)
\end{split}
\end{equation}
where $(\zeta_v(k))_{v\in V, k\in \Z}$ are real numbers. 
This model has been studied on the lattice $\Z^d$ for $p=1$ in \cite{BK94} 
under the assumption that $\eta=(\zeta_v(k))_{v\in V, k\in \Z}$ 
is a process of i.i.d. random variables.  
As main result the existence of infinite-volume Gibbs measures $\mu^0[\eta]$ obtained with 
zero-boundary conditions was shown for small disorder, large inverse temperature 
in lattice dimensions $d\geq 3$. The proof was based on a rigorous  
renormalization group analysis using multiscale cluster expansions.

Let us consider the above perturbed 
models on the regular tree and present 
two stability theorems. 

\subsection{Stability Theorems}

\begin{thm}\label{Theorem6} Let $p\geq 1$ 
and consider the 
$p$-SOS models (\ref{model-p-sos}) under the assumptions of Theorem \ref{Theorem1}
formulated for the model without disorder. 

Consider now the model 
in the additional presence of quenched random fields $\eta=(\eta_v)_{v\in V}$ with Hamiltonian \eqref{randomrandomsurface}. 
Then there is a strictly positive threshold $\delta^*>0$ such that for each $\eta$ satisfying
\begin{equation}
\begin{split}\label{randomfieldfinitespin2}
\sup_{v\in V}|\eta_v|\leq \delta^*
\end{split}
\end{equation}
at $\beta$ sufficiently large, 
there is an identifiable class of extremal  
Gibbs measures $\mu^{\omega^0}_\beta[\eta]$, concentrated on the stable ground states $\omega^0$, 
as described in Theorem \ref{Theorem1}. 
 \end{thm} 

{\bf Remark:}  Note that the case of a small homogeneous non-zero field is encluded. 
Note also that $\delta^*$ provided by the theorem depends on the parameters 
$d,{d_{{\max}}},M$ describing the sparsity and uniform bounds on 
the increments of the ground states $\omega^0$, which were discussed before. 

We turn to a corresponding stability result in the remaining classes of 
disordered models described above. 

\begin{thm}\label{Theorem7} 
Consider the $\Z$-valued $p$-SOS models (\ref{model-p-sos}) for $p>0$ 
under the assumptions of Theorem \ref{Theorem1}
formulated for the model without disorder, but in the presence of additional local perturbations 
of the form (\eqref{additionalsinglesite}). 

Alternatively consider the 
$\Z_q$-valued models under the assumptions of Theorem \ref{Theorem2} formulated 
for the model without disorder, but in the presence of additional local perturbations 
of the form \ref{randomrandomsurface}

Then, for both types of models, 
there is a strictly positive threshold $\epsilon^*>0$ such that for each choice 
of local potentials $\eta=(\Psi_{v}(\cdot))_{v \in V}$ satisfying  
\begin{equation}
\begin{split}\label{randomfieldfinitespin}
\sup_{v\in V}\sup_{
k,l\in \Omega_0}|\Psi_{v}(k)-\Psi_{v}(l)|\leq \epsilon^*
\end{split}
\end{equation}
there is an identifiable class of extremal  
Gibbs measures $\mu^{\omega^0}[\eta]$, concentrated on the stable ground states $\omega^0$, 
as described in Theorem \ref{Theorem1} and Theorem  \ref{Theorem2} respectively. 
\end{thm}

\subsection{Proofs via stability of excess energy estimates} 

To prove the last two stability theorems it turns out that we can
use exactly the same contour definitions and 
characterizations of ground states \textit{in spe}, as for the unperturbed models. However, in doing 
this we need to ensure that the expansions around the ground states 
$\omega^0$ which were identified in the homogeneous models,  
and their consequences, also stay valid for the $\eta$-perturbed models. 
This leads us to the study of excess energies in the perturbed models. 

Recall for this purpose the definition of stability of a ground state $\omega^0$  
with a constant $c>0$ given above in 
the two cases of finite or infinite local state space. 
We then have the following 
lemma on the stability in the models perturbed by a collection of 
local potentials $\eta=(\Psi_v)_{v\in V}$.

\lemma{Consider $\Z$-valued or $\Z_q$-valued models with additional local 
potentials of the type \eqref{additionalsinglesite}. \label{Lem5.1}
 Assume that $\omega_0\in \Omega$ is a ground state 
which is stable with a constant $c>0$ for the model with $\eta=0$. 
\begin{itemize}
\item If the single-site potentials obey the smallness condition 
\begin{equation}
\begin{split}\label{randomfieldfinitespin}
\sup_{v\in V}\sup_{
k,l\in \Omega_0}|\Psi_{v}(k)-\Psi_{v}(l)|=:\epsilon <c
\end{split}
\end{equation}
then $\omega_0\in \Omega$ is a {stable ground state with reduced constant $c-\epsilon>0$. }
 \item Consider specifically the random field random surface model \eqref{randomrandomsurface}
with $p\geq 1$. If
\begin{equation}
\begin{split}\label{randomfieldfinitespin2}
\sup_{v\in V}|\eta_v|=:\delta<c
\end{split}
\end{equation}
then $\omega_0\in \Omega$ is a {stable ground state 
 with reduced constant $c-\delta>0$}. 
  \end{itemize}

\proof The statements follow from spelling out the definitions in terms of the excess energies, 
and the triangle inequality. Note for the second case, that the excess energy
of $\omega$ relative to $\omega^0$ in the model \eqref{randomrandomsurface} has the lower 
bound $\sum_{v}(c|\omega_v-\omega_v^0|^p-\epsilon |\omega_v-\omega_v^0|)\geq 
\sum_{v}(c-\epsilon)|\omega_v-\omega_v^0|^p
$. The estimate works iff $p\geq 1$, which had 
to be already assumed before to have a well-defined model.  This  
explains the restriction to the case of convex interactions in the $p$-SOS models with random field disorder in the 
formulation of the lemma and of Theorem \ref{Theorem6}.  \endproof

Given  Lemma \ref{Lem5.1} the proof of Theorem \ref{Theorem6} and Theorem \ref{Theorem7}
is straightforward.

\end{document}